\documentclass[reqno,10pt,twoside]{article}
\usepackage{fullpage}
\usepackage{amssymb}
\usepackage{amsmath}
\usepackage{amsthm}
\numberwithin{equation}{section}
\newtheorem{theorem}{Theorem}

\newtheorem{remark}{Remark}

\newtheorem{lemma}{Lemma}
\newcommand{\Nzero}{\mathbb N_0}
\newcommand{\card}[1]{\left|#1\right|}

\begin{document}

\title{\Large On strongly minimal asymptotic bases of order $h$ generated by powers of $2$\thanks{This work is supported by the National Natural Science Foundation of China (Grant No. 12301003), the Anhui Provincial Natural Science Foundation (Grant No. 2308085QA02) and the University Natural Science Research Project of Anhui Province (Grant No. 2022AH050171).}}
\author{\large Shi-Qiang Chen\thanks{Corresponding author. E-mail addresses: csq20180327@163.com (S.-Q. Chen), 2663398161@qq.com (X.-Y. Yan).}\quad and\quad Xin-Yue Yan\\[1.5em]
\normalsize School of Mathematics and Statistics\\[1em]
\normalsize Anhui Normal University, Wuhu 241002, PR China}
\date{}
\maketitle
\vspace{-1.5em}

{\bf Abstract.} Let $h\geq2$ be an integer and let
$\Nzero=W_1\cup\cdots\cup W_h$ be a partition into infinite sets.
For $W\subseteq\Nzero$, let $A(W)$ denote the set of all finite nonempty
sums of distinct powers $2^w$ with $w\in W$. We give a
necessary and sufficient condition for
$A=A(W_1)\cup\cdots\cup A(W_h)$ to be a strongly minimal asymptotic
basis of order $h$. This provides a complete answer to a problem posed by Nathanson in 1988.



\section{Introduction}
Let $\Nzero$ denote the set of nonnegative integers. For $A\subseteq\Nzero$,
a real number $x\geq1$, and an integer $h\geq2$, put
\[
 A(x)=|A\cap[1,x]|,
 \qquad
 hA=\{a_1+\cdots+a_h:a_1,\ldots,a_h\in A\}.
\]
We call $A$ an asymptotic basis of order $h$ if $hA$ contains every
sufficiently large integer. An asymptotic basis $A$ of order $h$ is
minimal if no proper subset of $A$ is an asymptotic basis of order $h$.
For $a\in A$, put
\[
 E_a=hA\setminus h(A\setminus\{a\}),
 \qquad E_a(x)=|E_a\cap[1,x]|.
\]
Following Nathanson~\cite{N}, an asymptotic basis $A$ of order $h$ is
called \emph{strongly minimal} if, for every $a\in A$, there exists a
constant $c_a>0$ such that
\(E_a(x)\geq c_aA(x)^{h-1}\)
for all sufficiently large $x$. 

For a nonempty set $W\subseteq\Nzero$, let
\[
 \mathcal F^*(W)=\{F\subseteq W:0<|F|<\infty\},
 \qquad
 A(W)=\left\{\sum_{w\in F}2^w:F\in\mathcal F^*(W)\right\}.
\]
In 1988, Nathanson \cite{N} proved that, for
\[
 W_i=\{w\in\Nzero:w\equiv i-1\pmod h\}\qquad(1\leq i\leq h),
\]
the set $A=A(W_1)\cup\cdots\cup A(W_h)$ is a strongly minimal
asymptotic basis of order $h$. Nathanson~\cite{N} also posed the following problem.\\
{\bf Nathanson's Problem.} Let $\Nzero=W_1\cup\cdots\cup W_h$ be a partition. Under what
conditions is an asymptotic basis of the form
\[
 A=A(W_1)\cup\cdots\cup A(W_h)
\]
strongly minimal?

For $W\subseteq\Nzero$, write $W(x)=|W\cap[0,x)|$ for $x\geq0$.

In 2011, Chen and Chen~\cite{CC} established the following two results.\\
{\bf Theorem A.} Let $\Nzero=W_1\cup W_2$ be a partition with $0\in W_1$ such
that $W_1$ and $W_2$ are infinite. Then
$A=A(W_1)\cup A(W_2)$ is a minimal asymptotic basis of order $2$ if
and only if either $W_1$ contains no consecutive integers or $W_2$
contains consecutive integers or both.\\
{\bf Theorem B.} Let $\Nzero=W_1\cup W_2$ be a partition such that $W_1$ and
$W_2$ are infinite. Suppose that
$A=A(W_1)\cup A(W_2)$ is a minimal asymptotic basis of order $2$.
Then $A$ is strongly minimal if and only if $W_i(x)=(1/2)x+O(1)$ for $i=1,2$.

Consequently, they resolved Nathanson's problem for $h=2$.

In 2025, Chen~\cite{Chen2025} further investigated Nathanson's problem
and obtained the following result.\\
{\bf Theorem C.} Let $h\geq2$ be an integer, and let $r$ be the least
integer satisfying $r>\log h/\log 2$. Let
$\Nzero=W_1\sqcup\cdots\sqcup W_h$ be a partition into infinite sets
such that each $W_i$ contains $r$ consecutive integers and
$W_i(x)=x/h+O(1)$ for $1\leq i\leq h$. Then
$A=A(W_1)\cup\cdots\cup A(W_h)$ is a strongly minimal asymptotic basis
of order $h$.

For further related results on minimal asymptotic bases and binary
or $g$-adic constructions, see [3-13, 15-17].

Throughout the paper, fix an integer $h\geq2$ and a partition
\[
 \Nzero=W_1\sqcup\cdots\sqcup W_h
\]
with $0\in W_1$, where all the $W_i$ are infinite, and put
\[
 A_i=A(W_i),\qquad B_i=A(W_i)\cup\{0\}\qquad(1\leq i\leq h),
 \qquad A=\bigcup_{i=1}^h A_i.
\]
By Lemma~\ref{lema1}, $A$ is an asymptotic basis of
order $h$. For integers $m\geq0$ and $1\leq i\leq h$, put
\[
 W_i(m)=|W_i\cap[0,m)|,
 \qquad P_i(m)=\sum_{w\in W_i\cap[0,m)}2^w.
\]
Then
\begin{equation}\label{eq2}
 \sum_{i=1}^h W_i(m)=m,
 \qquad \sum_{i=1}^hP_i(m)=2^m-1.
\end{equation}
Let
\[
 \mathcal K_h=
 \left\{(k_1,\ldots,k_h)\in\mathbb Z_{\geq0}^h:
 k_1+\cdots+k_h=h\right\}.
\]
We answer Nathanson's problem for every $h\geq2$ by establishing the following necessary and sufficient condition.

\begin{theorem} \label{thm1} The set $A$ is a strongly minimal asymptotic basis of order $h$ if
and only if the following two conditions hold:

\begin{enumerate}
\item[(I)] For every $i\in\{1,\ldots,h\}$,
\[
 W_i(m)=\frac{m}{h}+O(1).
\]

\item[(II)] For every $i\in\{1,\ldots,h\}$, every $w\in W_i$, and every
$(k_1,\ldots,k_h)\in\mathcal K_h\setminus\{(1,\ldots,1)\}$, there exists
an integer $t\geq1$ such that
\begin{equation}\label{eq3}
 \sum_{j=1}^h k_jP_j(t)
 <
 \begin{cases}
 \displaystyle\sum_{j\neq i}P_j(t),&t\leq w,\\[2mm]
 \displaystyle\sum_{j\neq i}P_j(t)+2^w,&t>w.
 \end{cases}
\end{equation}
\end{enumerate}
\end{theorem}

\begin{remark}\label{rem1}
For $h=2$, Theorems A and B imply that condition~\emph{(II)}
can be replaced by the requirement that $A$ be a minimal asymptotic
basis of order $2$, provided that condition~\emph{(I)} is retained.
\end{remark}

\section{Lemmas}

\begin{lemma}(See \cite[Lemma 1(c)]{N}.)\label{lema1} Under the fixed notation of the Introduction,
\[
 A=A(W_1)\cup\cdots\cup A(W_h)
\]
is an asymptotic basis of order $h$.
\end{lemma}

\begin{lemma}(See \cite[Lemma 2]{N}.)\label{lema2} Let $T,s\geq1$ and $x_1,\ldots,x_T,y_1,\ldots,y_s\in\Nzero$ satisfy
\[
 \sum_{\rho=1}^{T}2^{x_\rho}
 =
 \sum_{r=1}^{s}2^{y_r},
 \qquad
 y_1<\cdots<y_s.
\]
The $x_\rho$ need not be distinct. Then there exists a partition
\[
 \{1,\ldots,T\}=J_1\sqcup\cdots\sqcup J_s
\]
such that
\[
 2^{y_r}=\sum_{\rho\in J_r}2^{x_\rho}
 \qquad(1\leq r\leq s).
\]
\end{lemma}

For each $n\in\Nzero$, define
\begin{equation}\label{eq4}
\begin{aligned}
 n&=\sum_{t\geq0}\varepsilon_t(n)2^t,
 \qquad \varepsilon_t(n)\in\{0,1\},\\
 S(n)&=\{t\in\Nzero:\varepsilon_t(n)=1\}.
\end{aligned}
\end{equation}

\begin{lemma}\label{lema3} Let
\[
 n=b_1+\cdots+b_q,
 \qquad b_\ell\in B_{j_\ell}\quad(1\leq\ell\leq q).
\]
If $n'$ is a nonnegative integer satisfying
\[
 {S}(n')\subseteq{S}(n),
\]
then there exist $b'_\ell\in B_{j_\ell}$ such that
\[
 n'=b'_1+\cdots+b'_q,
 \qquad
 {S}(b'_\ell)\subseteq{S}(b_\ell)
 \quad(1\leq\ell\leq q).
\]
\end{lemma}

\begin{proof}
If $n=0$, then $b_1=\cdots=b_q=n'=0$, and the choice
$b'_\ell=0$ proves the assertion. Hence assume $n>0$.
Write
\[
 S(n)=\{y_1,\ldots,y_s\},
 \qquad y_1<\cdots<y_s,
 \qquad n=\sum_{r=1}^{s}2^{y_r}.
\]
For every $1\leq\ell\leq q$, write
\(b_\ell=\sum_{x\in{S}(b_\ell)}2^x\).
Let
\[
 T=\sum_{\ell=1}^{q}\card{{S}(b_\ell)}.
\]
List all the binary terms above, with multiplicities, as
\[
 2^{x_1},\ldots,2^{x_T}.
\]
Then
\[
 \sum_{\rho=1}^{T}2^{x_\rho}
 =\sum_{\ell=1}^{q}b_\ell
 =n
 =\sum_{r=1}^{s}2^{y_r}.
\]
Let $\lambda(\rho)$ denote the index of the summand containing
$2^{x_\rho}$, for $1\leq\rho\leq T$. Then, for
$1\leq\ell\leq q$,
\[
 \{x_\rho:\lambda(\rho)=\ell\}={S}(b_\ell).
\]

By Lemma~\ref{lema2}, there exists a partition
\[
 \{1,\ldots,T\}=J_1\sqcup\cdots\sqcup J_s
\]
such that
\[
 2^{y_r}=\sum_{\rho\in J_r}2^{x_\rho}
 \qquad(1\leq r\leq s).
\]
By
\[
 {S}(n')\subseteq{S}(n)=\{y_1,\ldots,y_s\},
\]
let
\[
 R'=\{r\in\{1,\ldots,s\}:y_r\in{S}(n')\}.
\]
Then
\[
 n'
 =\sum_{r\in R'}2^{y_r}
 =\sum_{r\in R'}\sum_{\rho\in J_r}2^{x_\rho}.
\]
For $1\leq\ell\leq q$, define
\[
 b'_\ell
 =\sum_{r\in R'}
   \sum_{\substack{\rho\in J_r\\\lambda(\rho)=\ell}}2^{x_\rho}.
\]
For a fixed $\ell$, only some of the binary terms of $b_\ell$ are retained. Hence
\[
 {S}(b'_\ell)
 \subseteq
 \{x_\rho:\lambda(\rho)=\ell\}
 ={S}(b_\ell).
\]
Since $b_\ell\in B_{j_\ell}$ and
$S(b'_\ell)\subseteq S(b_\ell)$, we have
\[
 S(b'_\ell)\subseteq S(b_\ell)\subseteq W_{j_\ell},
 \qquad b'_\ell\in B_{j_\ell}
 \quad(1\leq\ell\leq q).
\]
Finally, since $J_1,\ldots,J_s$ are pairwise disjoint,
\begin{align*}
 \sum_{\ell=1}^{q}b'_\ell
 &=\sum_{\ell=1}^{q}
   \sum_{r\in R'}
   \sum_{\substack{\rho\in J_r\\\lambda(\rho)=\ell}}2^{x_\rho}\\
 &=\sum_{r\in R'}\sum_{\rho\in J_r}2^{x_\rho}\\
 &=\sum_{r\in R'}2^{y_r}\\
 &=n'.
\end{align*}
This completes the proof of the lemma.
\end{proof}

For $m\geq0$, let
\(K(m)=\max_{1\leq j\leq h}W_j(m)\).
By the uniqueness of binary representations,
\begin{equation}\label{eq5}
 A(2^m-1)
 =\sum_{j=1}^h\bigl(2^{W_j(m)}-1\bigr),
 \qquad
 K(m)\geq\frac mh.
\end{equation}

\section{Proof of sufficiency of Theorem \ref{thm1}}

\begin{lemma}\label{lema4} Let $M\geq1$, let $q_0,\ldots,q_{M-1}$ be nonnegative integers, and let $n$ be an integer with $0\leq n<2^M$. Then
\[
 n=\sum_{t=0}^{M-1}d_t2^t,
 \qquad
 0\leq d_t\leq q_t,
\]
holds for some integers $d_t$ if and only if
\begin{equation}\label{eq6}
 n\bmod2^r
 \leq
 \sum_{t=0}^{r-1}q_t2^t
 \qquad(1\leq r\leq M).
\end{equation}
\end{lemma}

\begin{proof}
Suppose first that there exist integers $d_0,\ldots,d_{M-1}$ satisfying
\[
 n=\sum_{t=0}^{M-1}d_t2^t,
 \qquad
 0\leq d_t\leq q_t.
\]
For $1\leq r\leq M$,
\begin{align*}
 n&=\sum_{t=0}^{r-1}d_t2^t
 +2^r\sum_{t=r}^{M-1}d_t2^{t-r},\\
 n\bmod2^r&=\left(\sum_{t=0}^{r-1}d_t2^t\right)\bmod2^r,\\
 0\leq n\bmod2^r
 &\leq\sum_{t=0}^{r-1}d_t2^t
 \leq\sum_{t=0}^{r-1}q_t2^t.
\end{align*}
This proves \eqref{eq6}.

Conversely, we prove by induction on $M$ that the conditions
\begin{equation}\label{eq7}
 \left\{
 \begin{aligned}
 0&\leq n\leq\sum_{t=0}^{M-1}q_t2^t,\\
 n\bmod2^r&\leq\sum_{t=0}^{r-1}q_t2^t
 \qquad(1\leq r<M),
 \end{aligned}
 \right.
\end{equation}
imply the existence of integers $d_0,\ldots,d_{M-1}$ with
\[
 n=\sum_{t=0}^{M-1}d_t2^t,
 \qquad
 0\leq d_t\leq q_t.
\]
For $M=1$, take $d_0=n\in[0,q_0]$.

Now let $M\geq2$ and assume the assertion for $M-1$. Write
\(n=Q2^{M-1}+R\),
where
\[
 Q=\left\lfloor\frac{n}{2^{M-1}}\right\rfloor,
 \qquad
 0\leq R<2^{M-1}.
\]
Define
\[
 d_{M-1}=\min\{q_{M-1},Q\},
 \qquad
 n'=n-d_{M-1}2^{M-1}.
\]
Clearly,
\[
 0\leq d_{M-1}\leq q_{M-1}.
\]

If $Q\leq q_{M-1}$, then
\begin{align*}
 d_{M-1}&=Q,\\
 n'&=R=n\bmod2^{M-1},\\
 0\leq n'&\leq\sum_{t=0}^{M-2}q_t2^t
 \qquad\text{by \eqref{eq7}}.
\end{align*}
If $Q>q_{M-1}$, then
\begin{align*}
 d_{M-1}&=q_{M-1},\\
 n'&=(Q-q_{M-1})2^{M-1}+R>0,\\
 n'&=n-q_{M-1}2^{M-1}
 \leq\sum_{t=0}^{M-2}q_t2^t
 \qquad\text{by \eqref{eq7}}.
\end{align*}
In either case,
\begin{equation}\label{eq8}
 0\leq n'\leq\sum_{t=0}^{M-2}q_t2^t.
\end{equation}

Since
\(n-n'=d_{M-1}2^{M-1}\),
for every $1\leq r<M-1$, we have
\(n'\equiv n\pmod{2^r}\).
Therefore,
\[
 n'\bmod2^r
 =n\bmod2^r
 \leq
 \sum_{t=0}^{r-1}q_t2^t
 \qquad(1\leq r<M-1).
\]
By \eqref{eq8} and the induction hypothesis, there exist integers
$d_0,\ldots,d_{M-2}$ such that
\[
 n'=\sum_{t=0}^{M-2}d_t2^t,
 \qquad
 0\leq d_t\leq q_t
 \quad(0\leq t<M-1).
\]
Consequently,
\begin{align*}
 n
 &=n'+d_{M-1}2^{M-1}\\
 &=\sum_{t=0}^{M-2}d_t2^t+d_{M-1}2^{M-1}\\
 &=\sum_{t=0}^{M-1}d_t2^t,
\end{align*}
and
\[
 0\leq d_t\leq q_t
 \qquad(0\leq t<M).
\]
Thus the stronger assertion holds for every $M$. For $0\leq n<2^M$, \eqref{eq6} implies
\[
 0\leq n=n\bmod2^M\leq\sum_{t=0}^{M-1}q_t2^t,
 \qquad
 n\bmod2^r\leq\sum_{t=0}^{r-1}q_t2^t\quad(1\leq r<M).
\]
Applying \eqref{eq7}, we obtain
\[
 n=\sum_{t=0}^{M-1}d_t2^t,\qquad
 0\leq d_t\leq q_t\quad(0\leq t<M).
\]
\end{proof}

\begin{lemma}\label{lema5} Fix $i\in\{1,\ldots,h\}$ and $w\in W_i$. The following two assertions are equivalent.

\begin{enumerate}
\item[(a)] For every
$(k_1,\ldots,k_h)\in\mathcal K_h\setminus\{(1,\ldots,1)\}$, there exists an integer $t\geq1$ such that
\begin{equation}\label{eq9}
 \sum_{j=1}^h k_jP_j(t)
 <
 \begin{cases}
 \displaystyle\sum_{j\neq i}P_j(t),&t\leq w,\\[2mm]
 \displaystyle\sum_{j\neq i}P_j(t)+2^w,&t>w.
 \end{cases}
\end{equation}

\item[(b)] There exists an integer $M>w$ such that, for every
$(k_1,\ldots,k_h)\in\mathcal K_h\setminus\{(1,\ldots,1)\}$,
\begin{equation}\label{eq10}
 2^w+\sum_{\substack{0\leq t<M\\t\notin W_i}}2^t
 \notin
 \sum_{j=1}^h k_jB_j.
\end{equation}
\end{enumerate}
\end{lemma}

\begin{proof}
Fix an integer
\[
 M>w
\]
and a vector
\[
 (k_1,\ldots,k_h)\in\mathcal K_h.
\]
For each $0\leq t<M$, since $W_1,\ldots,W_h$ partition $\Nzero$, there is a unique index
\[
 c(t)\in\{1,\ldots,h\}
\]
with
\(t\in W_{c(t)}\).
Let
\(q_t=k_{c(t)} \qquad(0\leq t<M)\).

Let $0\leq n<2^M$. We first establish the equivalence
\[
 n\in\sum_{j=1}^{h}k_jB_j
 \quad\text{if and only if}\quad
 n=\sum_{t=0}^{M-1}d_t2^t
 \quad\text{for some integers }0\leq d_t\leq q_t.
\]
Indeed, if
\(n\in\sum_{j=1}^{h}k_jB_j\),
then we may write
\[
 n=\sum_{j=1}^{h}\sum_{u=1}^{k_j}b_{j,u},
 \qquad
 b_{j,u}\in B_j.
\]
Since all summands are nonnegative,
\[
 0\leq b_{j,u}\leq n<2^M,
 \qquad S(b_{j,u})\subseteq[0,M)\cap\mathbb Z.
\]
For $0\leq t<M$, define
\[
 d_t
 =\card{\{u:1\leq u\leq k_{c(t)},\ t\in{S}(b_{c(t),u})\}}.
\]
Since $t$ belongs to no set other than $W_{c(t)}$,
\[
 0\leq d_t\leq k_{c(t)}=q_t.
\]
Moreover,
\begin{align*}
 n
 &=\sum_{j=1}^{h}\sum_{u=1}^{k_j}
   \sum_{t\in{S}(b_{j,u})}2^t\\
 &=\sum_{t=0}^{M-1}d_t2^t.
\end{align*}

Conversely, suppose that there exist integers $d_t$ satisfying
\[
 n=\sum_{t=0}^{M-1}d_t2^t,
 \qquad
 0\leq d_t\leq k_{c(t)}.
\]
For every $0\leq t<M$, choose a subset $U_t$ of
\[
 \{1,\ldots,k_{c(t)}\}
\]
with $d_t$ elements. For
$1\leq j\leq h$ and $1\leq u\leq k_j$, define
\[
 b_{j,u}
 =\sum_{\substack{0\leq t<M\\t\in W_j\\u\in U_t}}2^t.
\]
By construction,
\[
 S(b_{j,u})\subseteq W_j,
 \qquad b_{j,u}\in B_j.
\]
Moreover, $|U_t|=d_t$, so
\begin{align*}
 \sum_{j=1}^{h}\sum_{u=1}^{k_j}b_{j,u}
 &=\sum_{t=0}^{M-1}d_t2^t\\
 &=n.
\end{align*}
Thus
\[
 n\in\sum_{j=1}^{h}k_jB_j.
\]

By Lemma~\ref{lema4}, this representation exists if and only if
\[
 n\bmod2^r
 \leq
 \sum_{t=0}^{r-1}q_t2^t
 \qquad(1\leq r\leq M).
\]
Since $q_t=k_{c(t)}$ and $t\in W_{c(t)}$, we have
\begin{align*}
 \sum_{t=0}^{r-1}q_t2^t
 &=\sum_{t=0}^{r-1}k_{c(t)}2^t\\
 &=\sum_{j=1}^{h}k_j
   \sum_{t\in W_j\cap[0,r)}2^t\\
 &=\sum_{j=1}^{h}k_jP_j(r).
\end{align*}
Therefore,
\begin{equation}\label{eq11}
\begin{gathered}
n\in\sum_{j=1}^{h}k_jB_j
\quad\text{if and only if}\\[2pt]
n\bmod 2^r\leq\sum_{j=1}^{h}k_jP_j(r)
\qquad(1\leq r\leq M).
\end{gathered}
\end{equation}

Now take
\[
 n_M
 =2^w+\sum_{\substack{0\leq t<M\\t\notin W_i}}2^t.
\]
Since $w\in W_i$ and $w<M$, the binary positions in the expression above are distinct and all are less than $M$. Hence
\[
 0<n_M\leq\sum_{t=0}^{M-1}2^t=2^M-1<2^M.
\]
Fix $1\leq r\leq M$.

If $r\leq w$, then $2^w$ is divisible by $2^r$, and hence
\begin{align*}
 n_M\bmod2^r
 &=\sum_{\substack{0\leq t<r\\t\notin W_i}}2^t\\
 &=\sum_{j\neq i}\sum_{t\in W_j\cap[0,r)}2^t\\
 &=\sum_{j\neq i}P_j(r).
\end{align*}
If $r>w$, then the position $w$ remains among the lowest $r$ binary positions, and hence
\begin{align*}
 n_M\bmod2^r
 &=2^w+\sum_{\substack{0\leq t<r\\t\notin W_i}}2^t\\
 &=2^w+\sum_{j\neq i}P_j(r).
\end{align*}
Thus
\begin{equation}\label{eq12}
 n_M\bmod2^r
 =
 \begin{cases}
 \displaystyle\sum_{j\neq i}P_j(r),&r\leq w,\\[2mm]
 \displaystyle\sum_{j\neq i}P_j(r)+2^w,&r>w.
 \end{cases}
\end{equation}

By \eqref{eq11} and \eqref{eq12},
$n_M\notin\sum_{j=1}^{h}k_jB_j$ holds precisely when there is
an integer $1\leq r\leq M$ such that
\begin{equation}\label{eq13}
 \sum_{j=1}^{h}k_jP_j(r)
 <
 \begin{cases}
 \displaystyle\sum_{j\neq i}P_j(r),&r\leq w,\\[2mm]
 \displaystyle\sum_{j\neq i}P_j(r)+2^w,&r>w.
 \end{cases}
\end{equation}

Suppose that (a) holds. Since
\[
 \mathcal K_h
 =\{(k_1,\ldots,k_h)\in\mathbb Z_{\geq0}^h:
 k_1+\cdots+k_h=h\},
\]
we obtain
\[
 \card{\mathcal K_h}=\binom{2h-1}{h-1}<\infty.
\]
For every
\[
 (k_1,\ldots,k_h)
 \in\mathcal K_h\setminus\{(1,\ldots,1)\},
\]
choose, by (a), an integer
\[
 t(k_1,\ldots,k_h)\geq1
\]
such that \eqref{eq9} holds. Let
\[
 M
 =\max\Bigl(
 \{w+1\}\cup
 \{t(k_1,\ldots,k_h):
 (k_1,\ldots,k_h)\in
 \mathcal K_h\setminus\{(1,\ldots,1)\}\}
 \Bigr).
\]
Then
\(M>w\).
For each such vector, take
\(r=t(k_1,\ldots,k_h)\leq M\).
By \eqref{eq9}, \eqref{eq13} holds at this value of $r$. Therefore,
\[
 n_M\notin\sum_{j=1}^{h}k_jB_j.
\]
This holds simultaneously for all
$(k_1,\ldots,k_h)\in\mathcal K_h\setminus\{(1,\ldots,1)\}$
and hence (b) follows.

Conversely, suppose that (b) holds. Fix any
\[
 (k_1,\ldots,k_h)
 \in\mathcal K_h\setminus\{(1,\ldots,1)\}.
\]
By (b),
\[
 n_M\notin\sum_{j=1}^{h}k_jB_j.
\]
By \eqref{eq13}, there exists an integer
\(1\leq r\leq M\)
such that
\[
 \sum_{j=1}^{h}k_jP_j(r)
 <
 \begin{cases}
 \displaystyle\sum_{j\neq i}P_j(r),&r\leq w,\\[2mm]
 \displaystyle\sum_{j\neq i}P_j(r)+2^w,&r>w.
 \end{cases}
\]
Taking $t=r$ proves (a).
\end{proof}

\begin{proof}[Proof of sufficiency]
Assume that conditions~(I) and~(II) hold. By condition~(I),
\begin{equation}\label{eq14}
 W_j(m)=\frac mh+O(1)
 \qquad(1\leq j\leq h).
\end{equation}
Fix
\[
 a\in A_i,
 \qquad
 w\in{S}(a)\subseteq W_i.
\]
By Lemma~\ref{lema5}, there exists an integer $M>w$ such that for every
\[
 (k_1,\ldots,k_h)\in\mathcal K_h\setminus\{(1,\ldots,1)\},
\]
we have
\begin{equation}\label{eq15}
 2^w+\sum_{\substack{0\leq t<M\\t\notin W_i}}2^t
 \notin
 \sum_{j=1}^hk_jB_j.
\end{equation}
The strict prefix inequality in \eqref{eq13}, once satisfied at some $r\leq M$, is independent of the choice of larger $M$. Thus the exclusion \eqref{eq15} persists when $M$ is enlarged. We may therefore also assume that
\begin{equation}\label{eq16}
 {S}(a)\subseteq[0,M),
 \qquad
 W_j\cap[0,M)\neq\varnothing
 \quad(j\neq i).
\end{equation}

Take
\[
 m>M,
 \qquad
 F\subseteq[M,m)\setminus W_i,
\]
and let
\begin{equation}\label{eq17}
 n_F
 =a+
 \sum_{\substack{0\leq t<M\\t\notin W_i}}2^t
 +\sum_{t\in F}2^t.
\end{equation}
By
\[
 {S}(a)\subseteq W_i,
 \qquad
 F\subseteq[M,m)\setminus W_i,
\]
and ${S}(a)\subseteq[0,M)$, we obtain
\begin{align*}
 {S}(n_F)
 &={S}(a)
 \sqcup
 \bigl(([0,M)\cap\mathbb Z)\setminus W_i\bigr)
 \sqcup F,\\
 {S}(n_F)\cap W_i
 &={S}(a),\\
 0&<n_F<2^m.
\end{align*}

For $j\neq i$, let
\[
 b_j=P_j(M)+\sum_{t\in F\cap W_j}2^t.
\]
By
\[
 P_j(M)=\sum_{\substack{0\leq t<M\\t\in W_j}}2^t,
 \qquad
 W_j\cap[0,M)\neq\varnothing,
\]
we obtain
\begin{align*}
 {S}(b_j)
 &=\bigl(W_j\cap[0,M)\bigr)\sqcup(F\cap W_j)
 \subseteq W_j,\\
 {S}(b_j)&\neq\varnothing,
\end{align*}
and hence
\[
 b_j\in A_j\qquad(j\neq i).
\]
Moreover, since
\[
 ([0,M)\cap\mathbb Z)\setminus W_i
 =\bigsqcup_{j\neq i}(W_j\cap[0,M)),
\]
and
\[
 F=\bigsqcup_{j\neq i}(F\cap W_j),
\]
we have
\begin{align}
 n_F
 &=a+\sum_{j\neq i}
 \left(P_j(M)+\sum_{t\in F\cap W_j}2^t\right)
 \nonumber\\
 &=a+\sum_{j\neq i}b_j
 \in hA.
 \label{eq18}
\end{align}

Consider any representation with $h$ summands
\[
 n_F=c_1+\cdots+c_h,
 \qquad
 c_r\in A.
\]
For $1\leq j\leq h$, let
\(k_j=\card{\{r:c_r\in A_j\}}\).
Then
\begin{align*}
 k_j&\geq0,\\
 \sum_{j=1}^hk_j&=h,
\end{align*}
that is,
\[
 (k_1,\ldots,k_h)\in\mathcal K_h.
\]
Let
\[
 N_M=2^w+\sum_{\substack{0\leq t<M\\t\notin W_i}}2^t.
\]
By $w\in{S}(a)$ and the decomposition of ${S}(n_F)$ above,
\[
 {S}(N_M)
 =\{w\}\sqcup
 \bigl(([0,M)\cap\mathbb Z)\setminus W_i\bigr)
 \subseteq{S}(n_F).
\]
For each $r$, let $j_r$ be the unique index with
$c_r\in A_{j_r}\subseteq B_{j_r}$. Since
$S(N_M)\subseteq S(n_F)$, Lemma~\ref{lema3} gives
\[
 N_M=\sum_{r=1}^hc'_r,\qquad
 c'_r\in B_{j_r},\qquad
 S(c'_r)\subseteq S(c_r).
\]
In particular,
\begin{equation}\label{eq19}
 N_M\in\sum_{j=1}^hk_jB_j.
\end{equation}
This contradicts \eqref{eq15} when
$(k_1,\ldots,k_h)\neq(1,\ldots,1)$. Therefore,
\begin{equation}\label{eq20}
 (k_1,\ldots,k_h)=(1,\ldots,1).
\end{equation}

By \eqref{eq20}, after relabeling the above representation we may write
\[
 n_F=c^{(1)}+\cdots+c^{(h)},
 \qquad
 c^{(j)}\in A_j.
\]
Then
\[
 {S}(c^{(j)})\subseteq W_j
 \qquad(1\leq j\leq h).
\]
From
\[
 W_r\cap W_s=\varnothing\qquad(r\neq s),
\]
we obtain
\[
 {S}(c^{(r)})\cap{S}(c^{(s)})=\varnothing
 \qquad(r\neq s).
\]
Since each binary position occurs in at most one summand, their sum involves no binary carries. Uniqueness of binary representations now gives
\begin{align*}
 S(n_F)&=\bigsqcup_{j=1}^h S(c^{(j)}),\\
 S(c^{(i)})&=S(n_F)\cap W_i=S(a),\\
 c^{(i)}&=\sum_{t\in S(c^{(i)})}2^t
 =\sum_{t\in S(a)}2^t=a.
\end{align*}
In particular,
\begin{equation}\label{eq21}
 c^{(i)}=a.
\end{equation}
Hence every representation
\[
 n_F=c_1+\cdots+c_h,
 \qquad c_r\in A,
\]
satisfies
\[
 a\in\{c_1,\ldots,c_h\}.
\]
Combining this with \eqref{eq18}, we obtain
\[
 n_F\in hA,
 \qquad
 n_F\notin h(A\setminus\{a\}),
\]
and hence
\begin{equation}\label{eq22}
 n_F\in E_a.
\end{equation}

If
\(F_1\neq F_2\),
then by \eqref{eq17},
\[
 {S}(n_{F_1})\cap[M,m)=F_1
 \neq
 F_2={S}(n_{F_2})\cap[M,m).
\]
Thus $n_{F_1}\neq n_{F_2}$.
Moreover,
\begin{align}
 \card{[M,m)\setminus W_i}
 &=(m-M)-\card{W_i\cap[M,m)}
 \nonumber\\
 &=(m-M)-\bigl(W_i(m)-W_i(M)\bigr)
 \nonumber\\
 &=m-W_i(m)-M+W_i(M).
 \label{eq23}
\end{align}
Therefore, by \eqref{eq22},
\begin{align}
 E_a(2^m-1)
 &\geq
 \card{\{F:F\subseteq[M,m)\setminus W_i\}}
 \nonumber\\
 &=2^{\card{[M,m)\setminus W_i}}
 \nonumber\\
 &=2^{m-W_i(m)-M+W_i(M)}
 \nonumber\\
 &=2^{-M+W_i(M)}2^{m-W_i(m)}.
 \label{eq24}
\end{align}
By \eqref{eq14},
\begin{align}
 W_i(m)&=\frac mh+O(1),\nonumber\\
 m-W_i(m)&=\frac{h-1}{h}m+O(1).
 \label{eq25}
\end{align}
Substituting \eqref{eq25} into \eqref{eq24}, we obtain
\begin{equation}\label{eq26}
 E_a(2^m-1)\gg_a2^{(h-1)m/h}.
\end{equation}

By \eqref{eq5} and \eqref{eq14},
\begin{align*}
 A(2^m-1)
 &=\sum_{j=1}^h\bigl(2^{W_j(m)}-1\bigr)\\
 &=\sum_{j=1}^h2^{m/h+O(1)}-h\\
 &\asymp2^{m/h},
\end{align*}
Hence
\begin{equation}\label{eq27}
 A(2^m-1)\asymp2^{m/h}.
\end{equation}
By \eqref{eq26} and \eqref{eq27},
\begin{align}
 E_a(2^m-1)
 &\gg_a2^{(h-1)m/h}
 \nonumber\\
 &\asymp A(2^m-1)^{h-1},
 \label{eq28}
\end{align}
and hence
\[
 E_a(2^m-1)\gg_a A(2^m-1)^{h-1}.
\]

For sufficiently large $x$, choose an integer $m$ such that
\(2^m\leq x+1<2^{m+1}\).
Then
\(2^m-1\leq x<2^{m+1}-1\).
Therefore,
\begin{align*}
 E_a(x)
 &\geq E_a(2^m-1)\\
 &\gg_a2^{(h-1)m/h},
\end{align*}
and
\begin{align*}
 A(x)
 &\leq A(2^{m+1}-1)\\
 &\ll2^{(m+1)/h},
\end{align*}
Hence
\begin{align*}
 A(x)^{h-1}
 &\ll2^{(h-1)(m+1)/h}\\
 &=2^{(h-1)/h}2^{(h-1)m/h}\\
 &\ll_a E_a(x).
\end{align*}
Therefore,
\[
 E_a(x)\gg_a A(x)^{h-1}.
\]
Since
\(a\in A\)
was arbitrary, $A$ is a strongly minimal asymptotic basis of order $h$.
\end{proof}

\section{Proof of necessity of Theorem \ref{thm1}}

\begin{lemma}\label{lema6} Fix $i\in\{1,\ldots,h\}$ and $a\in A_i$. As $m\to\infty$,
\begin{equation}\label{eq29}
\begin{aligned}
&\card{\{0\leq n<2^m:n\in E_a,\ S(n)\cap W_i\neq S(a)\}}\\
&\quad=o\bigl(2^{(h-1)K(m)}\bigr).
\end{aligned}
\end{equation}
\end{lemma}

\begin{proof}
Fix an integer $L_a$ such that
\[
 {S}(a)\subseteq[0,L_a),
 \qquad
 a<2^{L_a}.
\]
Fix a sufficiently large integer $m>L_a$ and $n\in E_a\cap[0,2^m)$.
Since
\(E_a\subseteq hA\),
we may choose a representation
\[
 n=a_1+\cdots+a_h,
 \qquad
 a_r\in A\quad(1\leq r\leq h).
\]
Moreover, since
\(n\notin h(A\setminus\{a\})\),
at least one summand in this representation is equal to $a$. Let
\(R=\{r\in\{1,\ldots,h\}:a_r=a\}\).
Then
\(1\leq\card R\leq h\).

We first prove that if $s\notin R$ and $a_s\in A_i$, then
\begin{equation}\label{eq30}
 \card{\bigl({S}(a_s)\cap[L_a,m)\bigr)}\leq h-1.
\end{equation}
Suppose, for a contradiction, that
\[
 \card{\bigl({S}(a_s)\cap[L_a,m)\bigr)}\geq h.
\]
Since $s\notin R$,
\(\card R\leq h-1\).
Therefore, for every $r\in R$, we may choose distinct integers
\[
 v_r\in{S}(a_s)\cap[L_a,m).
\]
Since
\(a_s\in A_i\),
we obtain
\[
 {S}(a_s)\subseteq W_i,
 \qquad
 v_r\in W_i.
\]
Moreover, since
\[
 {S}(a)\subseteq[0,L_a),
 \qquad
 v_r\geq L_a,
\]
we have
\(v_r\notin{S}(a)\).
For $r\in R$, define
\(a'_r=a+2^{v_r}\).
Then
\[
 {S}(a'_r)={S}(a)\sqcup\{v_r\}\subseteq W_i,
\]
and hence
\[
 a'_r\in A_i,
 \qquad
 a'_r\neq a.
\]
Next define
\(a'_s=a_s-\sum_{r\in R}2^{v_r}\).
Since the $v_r$ are distinct elements of $S(a_s)$,
\[
 S(a'_s)=S(a_s)\setminus\{v_r:r\in R\},
\]
so that
\begin{align*}
 \card{\bigl({S}(a'_s)\cap[L_a,m)\bigr)}
 &\geq
 \card{\bigl({S}(a_s)\cap[L_a,m)\bigr)}-\card R\\
 &\geq h-(h-1)\\
 &=1.
\end{align*}
Therefore,
\[
 a'_s>0,
 \qquad
 {S}(a'_s)\subseteq W_i,
\]
and hence
\(a'_s\in A_i\).
Since $S(a'_s)\cap[L_a,m)\neq\varnothing$ whereas
$S(a)\cap[L_a,m)=\varnothing$, we also have $a'_s\neq a$.
For the remaining indices, define
\[
 a'_r=a_r
 \qquad(r\notin R\cup\{s\}).
\]
Since these indices are not in $R$,
\[
 a'_r=a_r\in A\setminus\{a\}.
\]
Consequently,
\[
 a'_r\in A\setminus\{a\}
 \qquad(1\leq r\leq h).
\]
On the other hand,
\begin{align*}
 \sum_{r=1}^{h}a'_r
 &=\sum_{r=1}^{h}a_r
   +\sum_{r\in R}2^{v_r}
   -\sum_{r\in R}2^{v_r}\\
 &=n.
\end{align*}
Thus $n\in h(A\setminus\{a\})$, contrary to $n\in E_a$.
This proves \eqref{eq30}. The same bound holds trivially for $a_r=a$, since
$S(a)\cap[L_a,m)=\varnothing$.

We estimate the number of elements of $A_i$ satisfying \eqref{eq30}. If
\(b\in A_i, \qquad 0<b<2^m\),
then
\({S}(b)\subseteq W_i\cap[0,m)\).
The choices for the two disjoint ranges satisfy
\[
 \begin{aligned}
 \#\{U:U\subseteq W_i\cap[0,L_a)\}&=2^{W_i(L_a)},\\
 \#\{U\subseteq W_i\cap[L_a,m):|U|\leq h-1\}
 &\leq\sum_{r=0}^{h-1}\binom{W_i(m)-W_i(L_a)}r.
 \end{aligned}
\]
Consequently,
\begin{equation}\label{eq31}
\begin{aligned}
&\card{\{b\in A_i:0<b<2^m,\ \card{S(b)\cap[L_a,m)}\leq h-1\}}\\
&\quad\leq 2^{W_i(L_a)}
\sum_{r=0}^{h-1}\binom{W_i(m)-W_i(L_a)}{r}\\
&\quad\leq 2^{W_i(L_a)}
\sum_{r=0}^{h-1}\binom{m-L_a}{r}\\
&\quad\leq C_{a,h}(m+1)^{h-1}.
\end{aligned}
\end{equation}

By \eqref{eq5},
\begin{equation}\label{eq32}
\begin{aligned}
A(2^m-1)
&=\sum_{j=1}^{h}\bigl(2^{W_j(m)}-1\bigr)\\
&\leq\sum_{j=1}^{h}2^{W_j(m)}
\leq h2^{K(m)}.
\end{aligned}
\end{equation}

Fix
\[
 \begin{gathered}
 0\leq n<2^m,\qquad n\in E_a,\qquad
 {S}(n)\cap W_i\neq{S}(a),\\
 n=\sum_{u=1}^{h}a_u,\qquad a_u\in A\ (1\leq u\leq h),
 \qquad a_{s_0}=a.
 \end{gathered}
\]
Let $N_1(m)$ and $N_2(m)$ denote the numbers of integers $n$
admitting representations of the following two types, respectively.

\noindent\emph{\bf Case 1.} Suppose that, for some $1\leq r\leq h-1$,
\[
 \card{\{u\neq s_0:a_u\in A_i\}}=r.
\]
By \eqref{eq30}, \eqref{eq31}, and
\eqref{eq32},
\[
 \begin{aligned}
 &\card{\{b\in A_i:0<b<2^m,\
 \ \card{{S}(b)\cap[L_a,m)}\leq h-1\}}
 \leq C_{a,h}(m+1)^{h-1},\\
 &\card{\{b\in A:0<b<2^m\}}
 =A(2^m-1)\leq h2^{K(m)},\\
 \end{aligned}
\]
The positions of $a$ and the $r$ remaining $A_i$-summands
admit at most $h\binom{h-1}{r}$ choices. Hence
\begin{align}
 N_1(m)
 &\leq h\sum_{r=1}^{h-1}\binom{h-1}{r}
 \bigl(C_{a,h}(m+1)^{h-1}\bigr)^r
 \bigl(h2^{K(m)}\bigr)^{h-1-r}
 \nonumber\\
 &\leq C'_{a,h}\sum_{r=1}^{h-1}
 (m+1)^{r(h-1)}2^{(h-1-r)K(m)}
 \nonumber\\
 &\leq C''_{a,h}(m+1)^{(h-1)^2}2^{(h-2)K(m)}.
 \label{eq33}
\end{align}
Since $K(m)\geq m/h$,
\begin{align}
 \frac{N_1(m)}{2^{(h-1)K(m)}}
 &\leq C''_{a,h}(m+1)^{(h-1)^2}2^{-K(m)}
 \nonumber\\
 &\leq C''_{a,h}(m+1)^{(h-1)^2}2^{-m/h}
 \longrightarrow0.
 \label{eq34}
\end{align}
Thus
\[
 N_1(m)=o\bigl(2^{(h-1)K(m)}\bigr).
\]

\noindent\emph{\bf Case 2.} Suppose that
\[
 a_u\notin A_i\qquad(u\neq s_0).
\]
For $j\neq i$, put
\[
 r_j=\card{\{u\neq s_0:a_u\in A_j\}},\qquad
 r_j\geq0,\qquad\sum_{j\neq i}r_j=h-1.
\]
If $(r_j)_{j\neq i}=(1,\ldots,1)$, let $b_j\in A_j$ be
the unique $A_j$-summand for $j\neq i$. Since their binary
supports are disjoint,
\[
 n=a+\sum_{j\neq i}b_j,
 \qquad
 S(n)\cap W_i=S(a),
\]
contrary to the choice of $n$. Hence
$(r_j)_{j\neq i}\neq(1,\ldots,1)$, and, since
$\sum_{j\neq i}r_j=h-1$, we have
\[
 \max_{j\neq i}r_j\geq2.
\]

Fix such a vector $(r_j)_{j\neq i}$ and write its summands as
$a_{j,1},\ldots,a_{j,r_j}\in A_j$. Since
$0<a_{j,u}\leq n<2^m$, for $j\neq i$ we have
\[
 \begin{aligned}
 \sum_{u=1}^{r_j}a_{j,u}
 &=\sum_{t\in W_j\cap[0,m)}
 \left(\sum_{u=1}^{r_j}\varepsilon_t(a_{j,u})\right)2^t,\\
 0&\leq\sum_{u=1}^{r_j}\varepsilon_t(a_{j,u})\leq r_j
 \qquad(t\in W_j\cap[0,m)).
 \end{aligned}
\]
Therefore,
\begin{equation}\label{eq35}
\begin{aligned}
N_2(m;(r_j)_{j\ne i})
&\leq \prod_{j\ne i}(r_j+1)^{W_j(m)}\\
&\leq \prod_{j\ne i}(r_j+1)^{K(m)}\\
&=2^{K(m)\sum_{j\ne i}\log_2(r_j+1)}.
\end{aligned}
\end{equation}
Moreover,
\[
 \begin{aligned}
 \log_2(r+1)&=r &&(r=0,1),\\
 r+1&\leq3\cdot2^{r-2} &&(r\geq2),\\
 \log_2(r+1)&\leq r-2+\log_2 3 &&(r\geq2).
 \end{aligned}
\]
Put
\begin{equation}\label{eq36}
 \delta=2-\log_2 3>0.
\end{equation}
Since $\max_{j\neq i}r_j\geq2$,
\[
 \sum_{j\neq i}\log_2(r_j+1)
 \leq\sum_{j\neq i}r_j-\delta
 =h-1-\delta.
\]
Thus
\[
 N_2(m;(r_j)_{j\neq i})
 \leq2^{(h-1-\delta)K(m)}.
\]
The number of possible vectors satisfies
\[
\begin{aligned}
&\#\left\{(r_j)_{j\ne i}\in\Nzero^{h-1}:
\begin{array}{l}
\sum_{j\ne i}r_j=h-1,\\[-1pt]
\max_{j\ne i}r_j\geq2
\end{array}\right\}\\
&\qquad=\binom{2h-3}{h-2}-1
\leq\binom{2h-3}{h-2}.
\end{aligned}
\]
Consequently,
\begin{equation}\label{eq37}
\begin{aligned}
N_2(m)
&\leq\binom{2h-3}{h-2}2^{(h-1-\delta)K(m)}\\
&=O_h\bigl(2^{(h-1-\delta)K(m)}\bigr)\\
&=o\bigl(2^{(h-1)K(m)}\bigr).
\end{aligned}
\end{equation}

Every integer under consideration occurs in at least one of the two cases. The cases need not be disjoint, so possible multiple representations only strengthen the following upper bound. Thus, by \eqref{eq34} and \eqref{eq37},
\[
\begin{aligned}
&\card{\{0\leq n<2^m:n\in E_a,\ S(n)\cap W_i\ne S(a)\}}\\
&\quad\leq N_1(m)+N_2(m)
=o\bigl(2^{(h-1)K(m)}\bigr).
\end{aligned}
\]
This is \eqref{eq29}.
\end{proof}

\begin{lemma}\label{lema7} If $A$ is a strongly minimal asymptotic basis of order $h$, then there exists a constant $C>0$ such that
\[
 \left|W_i(m)-\frac mh\right|\leq C
 \qquad(1\leq i\leq h,\ m\geq0).
\]
\end{lemma}

\begin{proof}
Since each $W_i$ is infinite, each $A_i$ is infinite. For
$1\leq i\leq h$, fix
\(a_i\in A_i\).
By strong minimality, for each $i$ there exist a constant $c_i>0$ and an integer $m_i$ such that, for
$m\geq m_i$,
\[
 E_{a_i}(2^m-1)
 \geq
 c_iA(2^m-1)^{h-1}.
\]
By \eqref{eq5},
\[
 A(2^m-1)
 =\sum_{j=1}^{h}\bigl(2^{W_j(m)}-1\bigr).
\]
Choose $j_0$ such that $W_{j_0}(m)=K(m)$. Then
\begin{align*}
 K(m)&\geq\frac1h\sum_{j=1}^hW_j(m)=\frac mh\longrightarrow\infty,\\
 A(2^m-1)&=\sum_{j=1}^h(2^{W_j(m)}-1)
 \geq 2^{K(m)}-1\geq2^{K(m)-1}
\end{align*}
for all sufficiently large $m$.
Hence, for every $i$, there exists a constant $c'_i>0$ such that
\begin{align}
 E_{a_i}(2^m-1)
 &\geq c_iA(2^m-1)^{h-1}
 \nonumber\\
 &\geq c_i\bigl(2^{K(m)-1}\bigr)^{h-1}
 \nonumber\\
 &=c'_i2^{(h-1)K(m)}
 \label{eq38}
\end{align}
for all sufficiently large $m$.

Fix $m$ with $S(a_i)\subseteq[0,m)$. Partition the elements of
$E_{a_i}\cap[0,2^m)$ according to whether $S(n)\cap W_i=S(a_i)$.
In the first class,
\[
\begin{gathered}
\varepsilon_t(n)=
\begin{cases}
1,&t\in S(a_i),\\
0,&t\in(W_i\cap[0,m))\setminus S(a_i),
\end{cases}\\[2pt]
\#\bigl(([0,m)\cap\mathbb Z)\setminus W_i\bigr)=m-W_i(m).
\end{gathered}
\]
Therefore
\begin{align*}
 \#\{0\leq n<2^m:S(n)\cap W_i=S(a_i)\}
 &=2^{m-W_i(m)},\\
 \#\{0\leq n<2^m:n\in E_{a_i},\ S(n)\cap W_i\neq S(a_i)\}
 &=o(2^{(h-1)K(m)})
\end{align*}
by Lemma~\ref{lema6}. Hence
\begin{equation}\label{eq39}
 E_{a_i}(2^m-1)
 \leq
 2^{m-W_i(m)}
 +o\bigl(2^{(h-1)K(m)}\bigr).
\end{equation}

Combining \eqref{eq38} and \eqref{eq39}, for each fixed $i$ we have
\[
 c'_i2^{(h-1)K(m)}
 \leq
 2^{m-W_i(m)}
 +o\bigl(2^{(h-1)K(m)}\bigr).
\]
For sufficiently large $m$,
\[
 o\bigl(2^{(h-1)K(m)}\bigr)
 \leq
 \frac{c'_i}{2}2^{(h-1)K(m)}.
\]
It follows that
\[
 2^{m-W_i(m)}
 \geq
 \frac{c'_i}{2}2^{(h-1)K(m)}.
\]
Since there are only finitely many indices $i$, we may choose a positive constant $c$, depending only on the $c'_i$, such that for every $i$ and all sufficiently large $m$,
\begin{equation}\label{eq40}
 2^{m-W_i(m)}
 \geq
 c\,2^{(h-1)K(m)}.
\end{equation}

Choose $i(m)$ with $W_{i(m)}(m)=K(m)$. Since the constant
$c$ in \eqref{eq40} is independent of $i$,
\begin{align*}
 2^{m-K(m)}&\geq c\,2^{(h-1)K(m)},
 \end{align*}
and so
\[ hK(m)\leq m-\log_2c.\]

Consequently,
\begin{equation}\label{eq41}
 K(m)\leq\frac mh+C_1
\end{equation}
for some constant $C_1>0$ and all sufficiently large $m$.

Since $W_i(m)\leq K(m)$,
\[
 W_i(m)\leq\frac mh+C_1
 \qquad(1\leq i\leq h).
\]
By \eqref{eq2},
\[
 \sum_{j=1}^{h}W_j(m)=m.
\]
Fix $i$. Then
\begin{align*}
 W_i(m)
 &=m-\sum_{j\neq i}W_j(m)\\
 &\geq
 m-\sum_{j\neq i}\left(\frac mh+C_1\right)\\
 &=m-(h-1)\frac mh-(h-1)C_1\\
 &=\frac mh-(h-1)C_1.
\end{align*}
Therefore, for all sufficiently large $m$,
\[
 \frac mh-(h-1)C_1
 \leq
 W_i(m)
 \leq
 \frac mh+C_1.
\]
Hence there exists a constant $C_2>0$ such that
\[
 \left|W_i(m)-\frac mh\right|\leq C_2
\]
for all $1\leq i\leq h$ and all sufficiently large $m$. Increasing the constant to cover the finitely many smaller values of $m$, we obtain $C>0$ such that
\[
 \left|W_i(m)-\frac mh\right|\leq C
 \qquad(1\leq i\leq h,\ m\geq0).
\]
This completes the proof of the lemma.
\end{proof}

\begin{lemma}\label{lema8} Let $a\in A_i$, let $I=[s,s+L)\cap\mathbb Z$, and suppose that
\(s>\max S(a)\).
Let $\mathcal P$ be an index set with $\card{\mathcal P}=h$, and suppose that
\[
 b_p\in A\cup\{0\}\qquad(p\in\mathcal P).
\]
Suppose that there exist $p_*\in\mathcal P$ and $\ell\in\{1,\ldots,h\}$ such that
\[
 b:=b_{p_*}\in A_\ell,
 \qquad
 \card{{S}(b)\cap I}\geq h.
\]
Put
\[
 R=\{p\in\mathcal P:b_p=a\},
 \qquad
 Z=\{p\in\mathcal P:b_p=0\}.
\]
If
\[
 R=\varnothing
 \qquad\text{or}\qquad
 \ell=i,
\]
then
\[
 \sum_{p\in\mathcal P}b_p\in h(A\setminus\{a\}).
\]
\end{lemma}

\begin{proof}
Since $S(b)\cap I\neq\varnothing$ whereas
$S(a)\cap I=\varnothing$, we have $b>0$ and $b\neq a$. Thus
\[
 p_*\notin R\cup Z,\qquad
 |R|+|Z|=|R\cup Z|\leq|\mathcal P|-1=h-1.
\]

By
\[
 \card{{S}(b)\cap I}\geq h
 \qquad\text{and}\qquad
 \card{R\cup Z}\leq h-1,
\]
we may choose, for every $p\in R\cup Z$, distinct integers
\(v_p\in{S}(b)\cap I\).
Let
\(V=\{v_p:p\in R\cup Z\}\).
Then
\[
 V\subseteq{S}(b)\cap I,
 \qquad
 \card V=\card R+\card Z\leq h-1.
\]
Define
\[
 b'_p=
 \begin{cases}
 a+2^{v_p},&p\in R,\\
 2^{v_p},&p\in Z,\\
 b-\displaystyle\sum_{v\in V}2^v,&p=p_*,\\
 b_p,&p\in\mathcal P\setminus(R\cup Z\cup\{p_*\}).
 \end{cases}
\]
We show that
\[
 b'_p\in A\setminus\{a\}.
\]

If $p\in R$, then $R\neq\varnothing$. By hypothesis,
\[
 R=\varnothing
 \qquad\text{or}\qquad
 \ell=i,
\]
we must have
\(\ell=i\).
Therefore,
\[
 v_p\in{S}(b)\subseteq W_i.
\]
Moreover, since
\(v_p\in I, \qquad s>\max S(a)\),
we have
\[
 v_p>\max S(a),
 \qquad
 v_p\notin{S}(a).
\]
Hence
\[
 {S}(b'_p)
 ={S}(a+2^{v_p})
 ={S}(a)\sqcup\{v_p\}
 \subseteq W_i.
\]
Consequently,
\(b'_p\in A_i\).
Furthermore,
\(b'_p=a+2^{v_p}>a\),
so
\[
 b'_p\in A_i\setminus\{a\}.
\]

For $p\in Z$, we have
\(v_p\in{S}(b)\subseteq W_\ell\),
and hence
\({S}(b'_p)=\{v_p\}\subseteq W_\ell\).
Therefore,
\[
 b'_p=2^{v_p}\in A_\ell.
\]
Since $v_p\in I$ and $S(a)\cap I=\varnothing$, we also have
$b'_p\neq a$. Thus $b'_p\in A\setminus\{a\}$.

For $p=p_*$, since
\(V\subseteq{S}(b)\),
we have
\[
 b'_{p_*}
 =b-\sum_{v\in V}2^v
 =\sum_{t\in{S}(b)\setminus V}2^t.
\]
Furthermore,
\begin{align*}
 \card{({S}(b)\cap I)\setminus V}
 &\geq
 \card{{S}(b)\cap I}-\card V\\
 &\geq h-(h-1)\\
 &=1.
\end{align*}
Therefore,
\[
 {S}(b)\setminus V\neq\varnothing,
 \qquad
 b'_{p_*}>0.
\]
Also,
\[
 {S}(b'_{p_*})={S}(b)\setminus V\subseteq W_\ell,
\]
and hence
\(b'_{p_*}\in A_\ell\).
Moreover, $S(b'_{p_*})\cap I\neq\varnothing$ whereas
$S(a)\cap I=\varnothing$. Thus $b'_{p_*}\in A\setminus\{a\}$.

Finally, take
\[
 p\in\mathcal P\setminus(R\cup Z\cup\{p_*\}).
\]
Since $p\notin Z$,
\(b_p\neq0\).
Using also
\(b_p\in A\cup\{0\}\),
we obtain
\(b_p\in A\).
Moreover, since $p\notin R$,
\(b_p\neq a\).
Therefore,
\[
 b'_p=b_p\in A\setminus\{a\}.
\]
Consequently, for every $p\in\mathcal P$,
\[
 b'_p\in A\setminus\{a\}.
\]

It remains to check the sum. Since $V=\{v_p:p\in R\cup Z\}$ and the $v_p$ are distinct,
\[
 \sum_{v\in V}2^v
 =\sum_{p\in R}2^{v_p}+\sum_{p\in Z}2^{v_p}.
\]
Therefore,
\begin{align*}
 \sum_{p\in\mathcal P}b'_p
 &=\sum_{p\in\mathcal P}b_p
   +\sum_{p\in R}2^{v_p}
   +\sum_{p\in Z}2^{v_p}
   -\sum_{v\in V}2^v\\
 &=\sum_{p\in\mathcal P}b_p.
\end{align*}
Hence
\[
 \sum_{p\in\mathcal P}b_p
 =\sum_{p\in\mathcal P}b'_p
 \in h(A\setminus\{a\}).
\]
This completes the proof of the lemma.
\end{proof}

\begin{lemma}\label{lema9} Suppose that, for some constant $C>0$,
\begin{equation}\label{eq42}
 \left|W_j(m)-\frac mh\right|\leq C
 \qquad(1\leq j\leq h,\ m\geq0).
\end{equation}
Then there exists an integer $L=L(h,C)\geq1$ with the following property. For every $i\in\{1,\ldots,h\}$, every $a\in A_i$, every interval $I=[s,s+L)\cap\mathbb Z$ with $s>\max S(a)$, and every
\[
 (k_1,\ldots,k_h)\in\mathcal K_h\setminus\{(1,\ldots,1)\},
\]
if an integer $n$ has binary digits $0$ at all positions in $W_i\cap I$ and binary digits $1$ at all positions in $I\setminus W_i$, then
\[
 n\in\sum_{j=1}^hk_jB_j
 \quad\Longrightarrow\quad
 n\in h(A\setminus\{a\}).
\]
\end{lemma}

\begin{proof}
For every interval $I=[s,s+L)\cap\mathbb Z$, where $s\geq0$ and $L\geq1$ are integers, put $m_j(I)=\card{W_j\cap I}$. By \eqref{eq42},
\begin{align}
 m_j(I)
 &=W_j(s+L)-W_j(s)
 \nonumber\\
 &=\frac Lh+e_j,
 \qquad
 |e_j|\leq2C.
 \label{eq43}
\end{align}
Let $D=2C$ and choose the integer
\begin{equation}\label{eq44}
 L=
 \left\lfloor
 \max\left\{
 h\bigl((2h+1)D+2(h-1)\bigr),
 h\bigl(2hD+h^2-1\bigr),
 \frac{hD}{h-1}+h(h+1)
 \right\}
 \right\rfloor+1.
\end{equation}
By this choice,
\begin{equation}\label{eq45}
 \frac Lh-D>2hD+2(h-1),
\end{equation}
\begin{equation}\label{eq46}
 \frac Lh-2hD-(h-1)>h(h-1),
\end{equation}
and
\begin{equation}\label{eq47}
 \frac{h-1}{h}L-D-(h-1)>h(h-1).
\end{equation}

Let
\[
 I=[s,s+L)\cap\mathbb Z,
 \qquad
 s>\max S(a),
\]
and suppose that
\[
 n\in\sum_{j=1}^hk_jB_j,
 \qquad
 (k_1,\ldots,k_h)\in\mathcal K_h\setminus\{(1,\ldots,1)\}.
\]
Since
\(\sum_{j=1}^hk_j=h\),
we may choose
\[
 b_{j,r}\in B_j
 \qquad
 (1\leq j\leq h,\ 1\leq r\leq k_j)
\]
such that
\begin{equation}\label{eq48}
 n=\sum_{j=1}^h\sum_{r=1}^{k_j}b_{j,r}.
\end{equation}
The summands $b_{j,r}$ may be $0$. Let
\[
 \mathcal P=
 \{(j,r):1\leq j\leq h,\ 1\leq r\leq k_j\}.
\]
Then
\begin{equation}\label{eq49}
 \card{\mathcal P}=\sum_{j=1}^hk_j=h.
\end{equation}

For $(j,r)\in\mathcal P$, since
\(b_{j,r}\in B_j=A(W_j)\cup\{0\}\)
we may write
\[
 b_{j,r}=\sum_{u\geq0}\varepsilon_u(b_{j,r})2^u,
\]
where
\begin{equation}\label{eq50}
 \varepsilon_u(b_{j,r})\in\{0,1\},
 \qquad
 \varepsilon_u(b_{j,r})=0\quad(u\notin W_j).
\end{equation}
Since
\(\Nzero=W_1\sqcup\cdots\sqcup W_h\),
for every $t\geq0$ there exists a unique
\[
 c(t)\in\{1,\ldots,h\}
\]
such that
\begin{equation}\label{eq51}
 t\in W_{c(t)}.
\end{equation}
Let
\begin{align}
 d_t
 &:=\sum_{j=1}^h\sum_{r=1}^{k_j}\varepsilon_t(b_{j,r})
 \nonumber\\
 &=\sum_{r=1}^{k_{c(t)}}\varepsilon_t(b_{c(t),r}).
 \label{eq52}
\end{align}
Then
\begin{align}
 0\leq d_t
 &\leq\sum_{r=1}^{k_{c(t)}}1
 =k_{c(t)}
 \leq\sum_{j=1}^hk_j
 =h.
 \label{eq53}
\end{align}
By \eqref{eq48},
\begin{align}
 n
 &=\sum_{j=1}^h\sum_{r=1}^{k_j}
   \sum_{u\geq0}\varepsilon_u(b_{j,r})2^u
 \nonumber\\
 &=\sum_{u\geq0}
   \left(\sum_{j=1}^h\sum_{r=1}^{k_j}\varepsilon_u(b_{j,r})\right)2^u
 \nonumber\\
 &=\sum_{u\geq0}d_u2^u.
 \label{eq54}
\end{align}
On the other hand, with the notation above, the standard binary expansion of $n$ is
\begin{equation}\label{eq55}
 n=\sum_{u\geq0}\varepsilon_u(n)2^u,
 \qquad
 \varepsilon_u(n)\in\{0,1\}.
\end{equation}
By the hypothesis of the lemma, for $t\in I$,
\begin{equation}\label{eq56}
 \varepsilon_t(n)=
 \begin{cases}
 0,&t\in W_i\cap I,\\
 1,&t\in I\setminus W_i.
 \end{cases}
\end{equation}

We now compare \eqref{eq54} and
\eqref{eq55} by binary carries. Put $c_0=0$ and, for
$t\geq0$, let
\[
 S_t=\sum_{u=0}^{t}d_u2^u,
 \qquad
 T_t=\sum_{u=0}^{t}\varepsilon_u(n)2^u.
\]
Both expansions represent $n$, so
\begin{equation}\label{eq57}
 S_t\equiv n\pmod{2^{t+1}},
\end{equation}
\begin{equation}\label{eq58}
 T_t\equiv n\pmod{2^{t+1}}.
\end{equation}
Since
\[
 S_t\geq0,
 \qquad
 0\leq T_t\leq2^{t+1}-1,
 \qquad
 S_t\equiv T_t\pmod{2^{t+1}},
\]
there exists a unique integer $c_{t+1}\geq0$ such that
\begin{equation}\label{eq59}
 \sum_{u=0}^{t}d_u2^u
 =
 \sum_{u=0}^{t}\varepsilon_u(n)2^u
 +2^{t+1}c_{t+1}.
\end{equation}
For $t\geq1$, substituting $t-1$ into
\eqref{eq59} gives
\begin{equation}\label{eq60}
 \sum_{u=0}^{t-1}d_u2^u
 =
 \sum_{u=0}^{t-1}\varepsilon_u(n)2^u
 +2^tc_t.
\end{equation}
Subtracting the two equalities,
\[
 d_t2^t
 =\varepsilon_t(n)2^t+2^{t+1}c_{t+1}-2^tc_t,
\]
and hence
\begin{equation}\label{eq61}
 d_t+c_t=\varepsilon_t(n)+2c_{t+1}
 \qquad(t\geq1).
\end{equation}
For $t=0$, the same identity follows from
\eqref{eq59} and $c_0=0$. Thus
\begin{equation}\label{eq62}
 d_t+c_t=\varepsilon_t(n)+2c_{t+1}
 \qquad(t\geq0).
\end{equation}

We next prove that the carries are uniformly bounded. We already have
\(c_0=0\leq h-1\).
If
\(0\leq c_t\leq h-1\),
then, by \eqref{eq62} and \eqref{eq53},
\begin{align*}
 0\leq2c_{t+1}
 &=d_t+c_t-\varepsilon_t(n)\\
 &\leq d_t+c_t\\
 &\leq h+(h-1)=2h-1.
\end{align*}
Since $c_{t+1}$ is an integer,
\[
 0\leq c_{t+1}
 \leq\left\lfloor\frac{2h-1}{2}\right\rfloor
 =h-1.
\]
By induction,
\begin{equation}\label{eq63}
\begin{gathered}
d_t+c_t=\varepsilon_t(n)+2c_{t+1},\\[2pt]
0\leq d_t\leq k_{c(t)},\qquad
0\leq c_t\leq h-1
\qquad(t\geq0).
\end{gathered}
\end{equation}

Summing
\eqref{eq63} over $t=s,\ldots,s+L-1$, we obtain
\begin{align*}
 \sum_{t=s}^{s+L-1}d_t
 +\sum_{t=s}^{s+L-1}c_t
 &=\sum_{t=s}^{s+L-1}\varepsilon_t(n)
 +2\sum_{t=s}^{s+L-1}c_{t+1}\\
 &=\sum_{t=s}^{s+L-1}\varepsilon_t(n)
 +2\sum_{t=s+1}^{s+L}c_t\\
 &=\sum_{t=s}^{s+L-1}\varepsilon_t(n)
 +2\sum_{t=s}^{s+L-1}c_t
 -2c_s+2c_{s+L}.
\end{align*}
Hence
\begin{equation}\label{eq64}
 \sum_{t\in I}c_t
 =\sum_{t\in I}(d_t-\varepsilon_t(n))
 +2c_s-2c_{s+L}.
\end{equation}

By \eqref{eq63} and $0\leq c_t\leq h-1$, we have
\begin{align*}
 \sum_{t\in I}(d_t-\varepsilon_t(n))
 &=\sum_{t=s}^{s+L-1}(2c_{t+1}-c_t)\\
 &=\sum_{t=s+1}^{s+L-1}c_t+2c_{s+L}-c_s\\
 &\geq-c_s\\
 &\geq-(h-1).
\end{align*}
Therefore,
\begin{equation}\label{eq72}
 \sum_{t\in I}d_t
 \geq\sum_{t\in I}\varepsilon_t(n)-(h-1).
\end{equation}

By \eqref{eq56},
\begin{equation}\label{eq73}
 \sum_{t\in I}\varepsilon_t(n)
 =\sum_{j\neq i}m_j(I).
\end{equation}

We distinguish three cases according to $k_i$.

\medskip
\noindent
\textbf{Case 1. $k_i=1$.} For $t\in W_i\cap I$, by \eqref{eq56} and
\eqref{eq53},
\[
 \varepsilon_t(n)=0,
 \qquad
 d_t\in\{0,1\}.
\]
By \eqref{eq63},
$d_t+c_t=2c_{t+1}$. If $d_t=0$, then $d_t\leq c_t$.
If $d_t=1$, then $c_t=2c_{t+1}-1\geq1=d_t$.
Thus
\begin{equation}\label{eq74}
 d_t\leq c_t
 \qquad(t\in W_i\cap I).
\end{equation}

By \eqref{eq64},
\begin{align*}
 &\sum_{t\in W_i\cap I}c_t
 +\sum_{t\in I\setminus W_i}c_t\\
 &=\sum_{t\in W_i\cap I}d_t
 +\sum_{t\in I\setminus W_i}d_t
 -\sum_{t\in I\setminus W_i}1
 +2c_s-2c_{s+L}.
\end{align*}
For $j\neq i$ and $t\in W_j\cap I$,
\(d_t\leq k_j\),
Hence
\[
 \sum_{t\in I\setminus W_i}d_t
 =\sum_{j\neq i}\sum_{t\in W_j\cap I}d_t
 \leq\sum_{j\neq i}k_jm_j(I).
\]
Moreover,
\[
 \sum_{t\in I\setminus W_i}1
 =\sum_{j\neq i}m_j(I),
\]
and
\[
 -2c_{s+L}\leq0,
 \qquad
 2c_s\leq2(h-1).
\]
Combining these with \eqref{eq74} and cancelling
$\sum_{t\in W_i\cap I}c_t$, we obtain
\begin{align}
 \sum_{t\in I\setminus W_i}c_t
 &\leq
 \sum_{j\neq i}(k_j-1)m_j(I)+2(h-1).
 \label{eq75}
\end{align}
Since $k_i=1$ and $\sum_{j=1}^hk_j=h$,
\begin{equation}\label{eq76}
 \sum_{j\neq i}k_j=h-1,
 \qquad
 \sum_{j\neq i}(k_j-1)=0.
\end{equation}
By \eqref{eq43}, write
\[
 m_j(I)=\frac Lh+e_j,
 \qquad |e_j|\leq D.
\]
Then \eqref{eq76} gives
\begin{align*}
 \sum_{j\neq i}(k_j-1)m_j(I)
 &=\frac Lh\sum_{j\neq i}(k_j-1)
 +\sum_{j\neq i}(k_j-1)e_j\\
 &=\sum_{j\neq i}(k_j-1)e_j\\
 &\leq D\sum_{j\neq i}|k_j-1|.
\end{align*}
Moreover, since
\[
 k_j\geq0,
 \qquad
 \sum_{j\neq i}k_j=h-1,
\]
we have
\begin{align*}
 \sum_{j\neq i}|k_j-1|
 &\leq\sum_{j\neq i}(k_j+1)\\
 &=\sum_{j\neq i}k_j+(h-1)\\
 &=2(h-1)<2h.
\end{align*}
Hence
\begin{equation}\label{eq77}
 \sum_{j\neq i}(k_j-1)m_j(I)\leq2hD.
\end{equation}
By \eqref{eq75} and \eqref{eq77},
\begin{equation}\label{eq78}
 \sum_{t\in I\setminus W_i}c_t
 \leq2hD+2(h-1).
\end{equation}

Since $k_i=1$, $\sum_{j\neq i}k_j=h-1$, but
$(k_1,\ldots,k_h)\neq(1,\ldots,1)$. Thus some index
$j_0\neq i$ satisfies $k_{j_0}=0$.
For $t\in W_{j_0}\cap I$, by
\eqref{eq53} and \eqref{eq56},
\[
 d_t=0,
 \qquad
 \varepsilon_t(n)=1.
\]
so \eqref{eq63} gives
\(c_t=1+2c_{t+1}\geq1\).
Therefore,
\begin{align}
 \sum_{t\in I\setminus W_i}c_t
 &\geq\sum_{t\in W_{j_0}\cap I}c_t
 \nonumber\\
 &\geq m_{j_0}(I)
 \nonumber\\
 &\geq\frac Lh-D.
 \label{eq79}
\end{align}
By \eqref{eq45},
\[
 \frac Lh-D>2hD+2(h-1).
\]
This contradicts \eqref{eq78} and \eqref{eq79}.
Therefore,
\begin{equation}\label{eq80}
 k_i\neq1.
\end{equation}

\medskip
\noindent
\textbf{Case 2. $k_i\geq2$.} By \eqref{eq53},
\begin{align*}
 \sum_{t\in I}d_t
 &=\sum_{t\in W_i\cap I}d_t
 +\sum_{j\neq i}\sum_{t\in W_j\cap I}d_t\\
 &\leq\sum_{t\in W_i\cap I}d_t
 +\sum_{j\neq i}k_jm_j(I).
\end{align*}
Therefore, by \eqref{eq72} and \eqref{eq73},
\begin{align*}
 \sum_{t\in W_i\cap I}d_t
 &\geq
 \sum_{j\neq i}m_j(I)-(h-1)
 -\sum_{j\neq i}k_jm_j(I)\\
 &=\sum_{j\neq i}(1-k_j)m_j(I)-(h-1).
\end{align*}
Moreover,
\[
 \sum_{j\neq i}(1-k_j)
 =(h-1)-(h-k_i)=k_i-1.
\]
Using
\[
 m_j(I)=\frac Lh+e_j,
 \qquad |e_j|\leq D,
\]
we obtain
\begin{align*}
 \sum_{j\neq i}(1-k_j)m_j(I)
 &=\frac Lh(k_i-1)
 +\sum_{j\neq i}(1-k_j)e_j\\
 &\geq\frac Lh(k_i-1)
 -D\sum_{j\neq i}|1-k_j|.
\end{align*}
Also,
\begin{align*}
 \sum_{j\neq i}|1-k_j|
 &\leq\sum_{j\neq i}(1+k_j)\\
 &=(h-1)+(h-k_i)\\
 &\leq2h.
\end{align*}
Therefore,
\begin{align}
 \sum_{t\in W_i\cap I}d_t
 &\geq
 (k_i-1)\frac Lh-2hD-(h-1)
 \nonumber\\
 &\geq
 \frac Lh-2hD-(h-1)
 \nonumber\\
 &>h(h-1),
 \label{eq81}
\end{align}
where the last step follows from \eqref{eq46}.

For $1\leq r\leq k_i$, let
\(q_r=\card{{S}(b_{i,r})\cap I}\).
By \eqref{eq52},
\begin{align*}
 \sum_{r=1}^{k_i}q_r
 &=\sum_{r=1}^{k_i}\sum_{t\in W_i\cap I}\varepsilon_t(b_{i,r})\\
 &=\sum_{t\in W_i\cap I}\sum_{r=1}^{k_i}\varepsilon_t(b_{i,r})\\
 &=\sum_{t\in W_i\cap I}d_t\\
 &>h(h-1).
\end{align*}
Since $k_i\leq h$, the inequality above gives
$\sum_{r=1}^{k_i}q_r>k_i(h-1)$. Thus there exists
$r_*$ such that
\begin{equation}\label{eq82}
 b:=b_{i,r_*}\in A_i,
 \qquad
 \card{{S}(b)\cap I}\geq h.
\end{equation}
Let
\(p_*=(i,r_*)\).
By \eqref{eq48}, \eqref{eq49},
\eqref{eq82}, and Lemma~\ref{lema8} (with
$\ell=i$), we obtain
in the case $k_i\geq2$,
\begin{equation}\label{eq83}
 n\in h(A\setminus\{a\}).
\end{equation}

\medskip
\noindent
\textbf{Case 3. $k_i=0$.} By \eqref{eq72} and \eqref{eq73},
\begin{align*}
 \sum_{t\in I}d_t
 &\geq\sum_{j\neq i}m_j(I)-(h-1)\\
 &=L-m_i(I)-(h-1).
\end{align*}
By \eqref{eq43},
\[
 m_i(I)=\frac Lh+e_i,
 \qquad |e_i|\leq D,
\]
and hence
\begin{align}
 \sum_{t\in I}d_t
 &\geq L-\left(\frac Lh+D\right)-(h-1)
 \nonumber\\
 &=\frac{h-1}{h}L-D-(h-1)
 \nonumber\\
 &>h(h-1),
 \label{eq84}
\end{align}
where the last step follows from \eqref{eq47}.

For every $p=(j,r)\in\mathcal P$, let
\(q_p=\card{{S}(b_p)\cap I}\).
By \eqref{eq52},
\begin{align*}
 \sum_{p\in\mathcal P}q_p
 &=\sum_{j=1}^h\sum_{r=1}^{k_j}
   \sum_{t\in I}\varepsilon_t(b_{j,r})\\
 &=\sum_{t\in I}
   \sum_{j=1}^h\sum_{r=1}^{k_j}\varepsilon_t(b_{j,r})\\
 &=\sum_{t\in I}d_t\\
 &>h(h-1).
\end{align*}
Since $|\mathcal P|=h$, the inequality above gives
$\sum_{p\in\mathcal P}q_p>|\mathcal P|(h-1)$. Thus
there exists $p_*=(j_*,r_*)\in\mathcal P$ such that
\begin{equation}\label{eq85}
 b:=b_{p_*}\in A_{j_*},
 \qquad
 \card{{S}(b)\cap I}\geq h.
\end{equation}
Since $k_i=0$,
\begin{equation}\label{eq86}
 j_*\neq i.
\end{equation}
For all $p=(j,r)\in\mathcal P$, we have $j\neq i$, and hence
\[
 b_p\in B_j=A_j\cup\{0\},\qquad
 (A_j\cup\{0\})\cap A_i=\varnothing,
 \qquad a\in A_i.
\]
Thus
\begin{equation}\label{eq87}
 \{p\in\mathcal P:b_p=a\}=\varnothing.
\end{equation}
By \eqref{eq48}, \eqref{eq49},
\eqref{eq85}--\eqref{eq87}, and
Lemma~\ref{lema8} (with $\ell=j_*$), in the case $k_i=0$ we obtain
\begin{equation}\label{eq88}
 n\in h(A\setminus\{a\}).
\end{equation}

Since $k_i\in\{0,1,\ldots,h\}$, the case $k_i=1$ is excluded by
\eqref{eq80}. By
\eqref{eq83} and \eqref{eq88}, in each of the other two cases we have
\(n\in h(A\setminus\{a\})\).
This completes the proof of the lemma.
\end{proof}

\begin{proof}[Proof of necessity]
Assume that $A$ is a strongly minimal asymptotic basis of order $h$. By Lemma~\ref{lema7}, fix a constant $C>0$ such that
\[
 \left|W_j(m)-\frac mh\right|\leq C
 \qquad(1\leq j\leq h,\ m\geq0).
\]
In particular,
\begin{equation}\label{eq89}
 W_i(m)=\frac mh+O(1)
 \qquad(1\leq i\leq h).
\end{equation}
In particular,
\[
 K(m)=\max_{1\leq j\leq h}W_j(m)
 =\frac mh+O(1).
\]

Suppose that \eqref{eq3} fails. Then there exist
\[
 i\in\{1,\ldots,h\},
 \qquad
 w\in W_i,
\]
and
\[
 (k_1,\ldots,k_h)
 \in\mathcal K_h\setminus\{(1,\ldots,1)\}
\]
such that, for every $t\geq1$,
\begin{equation}\label{eq90}
 \sum_{j=1}^hk_jP_j(t)
 \geq
 \begin{cases}
 \displaystyle\sum_{j\neq i}P_j(t),&t\leq w,\\[2mm]
 \displaystyle\sum_{j\neq i}P_j(t)+2^w,&t>w.
 \end{cases}
\end{equation}
Let
\[
 a=2^w.
\]
Then
\(a\in A_i, \qquad {S}(a)=\{w\}\).
For any $M>w$, let
\[
 N_M=2^w+\sum_{\substack{0\leq t<M\\t\notin W_i}}2^t.
\]
By \eqref{eq90} and \eqref{eq12}, for every
$M>w$ and $1\leq r\leq M$ we have
\[
 N_M\bmod2^r\leq\sum_{j=1}^hk_jP_j(r).
\]
Since $0<N_M<2^M$, the representation criterion
\eqref{eq11} gives
\begin{equation}\label{eq91}
 N_M\in\sum_{j=1}^hk_jB_j.
\end{equation}

Fix $m>w$ and let
\[
 \mathcal H_m
 =\{0\leq n<2^m:{S}(n)\cap W_i=\{w\}\}.
\]
Let
\[
 \Omega_m=([0,m)\cap\mathbb Z)\setminus W_i.
\]
Then
\(\card{\Omega_m}=m-W_i(m)\).
For every $S\subseteq\Omega_m$, let
\(n_S=2^w+\sum_{t\in S}2^t\).
As $S(n_S)=\{w\}\sqcup S$, distinct subsets of $\Omega_m$
give distinct integers $n_S$. Moreover,
$\mathcal H_m=\{n_S:S\subseteq\Omega_m\}$.
Hence
\begin{equation}\label{eq92}
 |\mathcal H_m|=2^{|\Omega_m|}=2^{m-W_i(m)}.
\end{equation}

For every $n\in\mathcal H_m$,
\begin{align*}
 {S}(n)
 &\subseteq
 \{w\}\sqcup\Omega_m\\
 &={S}(N_m).
\end{align*}
Applying \eqref{eq91} with $M=m$ and then Lemma~\ref{lema3}, we obtain
\begin{equation}\label{eq93}
 n\in\sum_{j=1}^hk_jB_j
 \qquad(n\in\mathcal H_m).
\end{equation}

Let $L=L(h,C)$ be the integer in Lemma~\ref{lema9}. Let
\[
 q=\left\lfloor\frac{m-w-1}{L}\right\rfloor.
\]
For sufficiently large $m$, we have $q\geq1$. Define
\[
 I_r=[w+1+(r-1)L,w+1+rL)\cap\mathbb Z
 \qquad(1\leq r\leq q).
\]
Then
\begin{align*}
 \card{I_r}&=L,\\
 I_r\cap I_s&=\varnothing\qquad(r\neq s),\\
 \min I_r&>w=\max S(a),\\
 I_r&\subseteq[0,m).
\end{align*}
For $1\leq r\leq q$, let
\[
 J_r=I_r\setminus W_i,
 \qquad
 \ell_r=\card{J_r}.
\]
Then
\[
 0\leq\ell_r\leq L,
 \qquad
 J_r\cap J_s=\varnothing\quad(r\neq s).
\]
For $n\in\mathcal H_m$, let $\mathcal C_r(n)$ mean that
$S(n)\cap J_r=J_r$. Since $S(n)\cap W_i=\{w\}$ and
$\min I_r>w$, all binary digits of $n$ on $W_i\cap I_r$
are $0$. If $\mathcal C_r(n)$ holds, all digits on
$J_r=I_r\setminus W_i$ are $1$. Thus
\begin{equation}\label{eq94}
 \varepsilon_t(n)=
 \begin{cases}
 0,&t\in W_i\cap I_r,\\
 1,&t\in I_r\setminus W_i.
 \end{cases}
\end{equation}
By \eqref{eq93},
\[
 n\in\sum_{j=1}^hk_jB_j,
 \qquad
 (k_1,\ldots,k_h)\neq(1,\ldots,1).
\]
Therefore, whenever $\mathcal C_r(n)$ holds, Lemma~\ref{lema9} together with
\eqref{eq94} gives
\begin{equation}\label{eq95}
 n\in h(A\setminus\{a\}).
\end{equation}

Now take
\(n\in E_a\cap\mathcal H_m\).
By $E_a=hA\setminus h(A\setminus\{a\})$,
\[
 n\notin h(A\setminus\{a\}).
\]
By \eqref{eq95}, $\mathcal C_r(n)$ fails for every
$1\leq r\leq q$. Consequently,
\begin{equation}\label{eq1}
 E_a\cap\mathcal H_m
 \subseteq
 \{n\in\mathcal H_m:
 \mathcal C_r(n)\text{ fails for every }1\leq r\leq q\}.
\end{equation}

For each $r$, the $\ell_r$ digits in $J_r$ admit $2^{\ell_r}$
assignments, exactly one of which satisfies $\mathcal C_r$. Hence
\[
 \#\{(\varepsilon_t)_{t\in J_r}:\mathcal C_r(n)\text{ fails}\}
 =2^{\ell_r}-1.
\]
Therefore,
\begin{align*}
 \frac{2^{\ell_r}-1}{2^{\ell_r}}
 &=1-2^{-\ell_r}\\
 &\leq1-2^{-L}.
\end{align*}
Let
\(J=\bigsqcup_{r=1}^qJ_r\).
Then
\begin{align*}
 \card J&=\sum_{r=1}^q\ell_r,\\
 \card{\Omega_m\setminus J}
 &=m-W_i(m)-\sum_{r=1}^q\ell_r.
\end{align*}
Therefore,
\[
\begin{aligned}
&\card{\{n\in\mathcal H_m:
\mathcal C_r(n)\text{ fails for every }1\leq r\leq q\}}\\
&\quad=2^{\card{\Omega_m\setminus J}}
\prod_{r=1}^q(2^{\ell_r}-1)\\
&\quad=2^{m-W_i(m)-\sum_{r=1}^q\ell_r}
\prod_{r=1}^q(2^{\ell_r}-1)\\
&\quad=2^{m-W_i(m)}\prod_{r=1}^q(1-2^{-\ell_r})\\
&\quad\leq2^{m-W_i(m)}(1-2^{-L})^q.
\end{aligned}
\]
By \eqref{eq1},
\begin{align}
 \card{E_a\cap\mathcal H_m}
 &\leq
 2^{m-W_i(m)}(1-2^{-L})^q
 \nonumber\\
 &=
 2^{m-W_i(m)}
 (1-2^{-L})^{\lfloor(m-w-1)/L\rfloor}.
 \label{eq96}
\end{align}

Let
\[
 \rho=1-2^{-L}\in(0,1),
 \qquad
 \theta=\rho^{1/L}\in(0,1).
\]
By
\[
 q=\left\lfloor\frac{m-w-1}{L}\right\rfloor
 \geq\frac{m-w-1}{L}-1,
\]
we obtain
\begin{align*}
 \rho^q
 &\leq
 \rho^{(m-w-1)/L-1}\\
 &=\rho^{-1}\theta^{m-w-1}\\
 &=O_{w,L}(\theta^m)\\
 &=o(1).
\end{align*}
By \eqref{eq89},
\begin{align*}
 m-W_i(m)
 &=\frac{h-1}{h}m+O(1),\\
 2^{m-W_i(m)}
 &\asymp2^{(h-1)m/h}.
\end{align*}
Substituting these two estimates into \eqref{eq96}, we obtain
\begin{equation}\label{eq97}
 \card{E_a\cap\mathcal H_m}
 =o\bigl(2^{(h-1)m/h}\bigr).
\end{equation}

On the other hand, by Lemma~\ref{lema6},
\[
 \card{(E_a\cap[0,2^m))\setminus\mathcal H_m}
 =o\bigl(2^{(h-1)K(m)}\bigr).
\]
Combining this with $K(m)=m/h+O(1)$ yields
\[
 2^{(h-1)K(m)}
 =2^{(h-1)m/h+O(1)}\asymp2^{(h-1)m/h},
\]
so that
\begin{equation}\label{eq98}
 \card{(E_a\cap[0,2^m))\setminus\mathcal H_m}
 =o\bigl(2^{(h-1)m/h}\bigr).
\end{equation}
Moreover, since
\[
 E_a\cap[0,2^m)
 =
 (E_a\cap\mathcal H_m)
 \sqcup
 \bigl((E_a\cap[0,2^m))\setminus\mathcal H_m\bigr),
\]
we have, by \eqref{eq97} and \eqref{eq98},
\begin{align}
 E_a(2^m-1)
 &=\card{E_a\cap[0,2^m)}
 \nonumber\\
 &=o\bigl(2^{(h-1)m/h}\bigr).
 \label{eq99}
\end{align}

On the other hand, strong minimality gives a constant $c_a>0$ such that, for all sufficiently large $m$,
\[
 E_a(2^m-1)
 \geq
 c_aA(2^m-1)^{h-1}.
\]
By \eqref{eq5} and \eqref{eq89},
$A(2^m-1)\asymp2^{m/h}$. Hence
\begin{align}
 E_a(2^m-1)
 &\gg_a A(2^m-1)^{h-1}
 \nonumber\\
 &\asymp2^{(h-1)m/h}.
 \label{eq100}
\end{align}
Therefore,
\[
 \frac{E_a(2^m-1)}{2^{(h-1)m/h}}
 \longrightarrow0,
\]
which contradicts
\[
 \frac{E_a(2^m-1)}{2^{(h-1)m/h}}
 \gg_a1.
\]
This contradiction establishes \eqref{eq3} for every
choice of $i$, $w$, and $(k_1,\ldots,k_h)$, and hence proves condition~(II).

Condition~(I) follows from \eqref{eq89}. 
\end{proof}

\end{document}